\documentclass[11pt,reqno]{amsart}
\usepackage{graphicx}
\usepackage{color}
\usepackage{amsmath,amssymb,amsthm,amsfonts}
\usepackage{graphicx}
\usepackage{color}
\usepackage{multicol}
\def\R{\mathbb{R}}
\makeatletter

\newcommand{\Rmnum}[1]{\expandafter\@slowromancap\romannumeral #1@}
\makeatother

\newtheorem{thm}{Theorem}[section]
\newcommand{\norm}[1]{\left\lVert#1\right\rVert}

\newtheorem{lemma}[thm]{Lemma}

\newtheorem{theorem}[thm]{Theorem}

\newcommand{\abs}[1]{\left\vert#1\right\vert}
\usepackage[numbers,sort&compress]{natbib}
\begin{document}
\author{Hai-Yang Jin}
\address{Department of Mathematics, South China University of Technology, Guangzhou 510640, China}
\email{mahyjin@scut.edu.cn}

\author{Zhi-An Wang}
\address{Department of Applied Mathematics, Hong Kong Polytechnic University, Hung Hom,
 Hong Kong}
\email{mawza@polyu.edu.hk}

\title[Critical mass on the Keller-Segel system with signal-dependent motility]{Critical mass on the Keller-Segel system with signal-dependent motility}

\begin{abstract}
This paper is concerned with the global boundedness and blowup of solutions to the Keller-Segel system with density-dependent motility in a two-dimensional bounded smooth domain with Neumman boundary conditions. We show that if the motility function decays exponentially, then a critical mass phenomenon similar to the minimal Keller-Segel model will arise. That is there is a number $m_*>0$, such that the solution will globally exist with uniform-in-time bound if the initial cell mass (i.e. $L^1$-norm of the initial value of cell density) is less than $m_*$, while the solution may blow up if  the initial cell mass is greater than $m_*$.

%We consider the original Keller-Segel system with signal-dependent motility \cite{KS-1971-JTB2}:
%\begin{equation*}
%\begin{cases}
%u_t=\nabla \cdot(\gamma(v)\nabla u-u\phi(v)\nabla v), &x\in \Omega, ~~t>0,\\
%v_t=\Delta v+u- v,& x\in \Omega, ~~t>0,\\
%u(x,0)=u_0(x),~v(x,0)=v_0(x), & x\in \Omega,
%\end{cases}
%\end{equation*}
%in a bounded domain $\Omega\subset \R^2$ with smooth boundary subject to homogeneous Neumann boundary conditions, where \begin{equation*}\label{KS-1}
%\phi(v)=(\alpha-1)\gamma'(v).
%\end{equation*}
%When $\alpha=0$ and $\gamma(v)=e^{-\chi v}$ with $\chi>0$, by constructing a  Lyapunov functional, {\color{black} we find a critical mass $m_*=\frac{4\pi}{\chi}$ in the sense that
%\begin{itemize}
%\item If $\int_\Omega u_0 dx<m_*$, the globally bounded solution exists;
%\item For any $M>m_*$ and $M\not \in \{\frac{4\pi m}{\chi}: m\in\mathbb{N}^+\}$ with  $\mathbb{N}^+$ denotes the set of positive integers,  there exist initial data $(u_0,v_0)$ satisfying $\int_{\Omega} u_0dx=M$ such that the corresponding solution blows up.
%\end{itemize}}

\end{abstract}

\subjclass[2000]{35A01, 35B44, 35K57, 35Q92, 92C17}

\keywords{Signal-Dependent motility, global existence, blow-up, critical mass}

   \maketitle

\numberwithin{equation}{section}
\section{Introduction}
To show how individual cell paths can result in an average cell flux proportional to the macroscopic chemical gradient,  Keller and Segel derived the following system based on a Brownian motion model of chemotaxis model in their seminal work \cite{KS-1971-JTB2}:
\begin{equation}\label{KS}
\begin{cases}
u_t=\nabla \cdot(\gamma(v)\nabla u-u\phi(v)\nabla v),\\
v_t=\Delta v+u-v,\\
\end{cases}
\end{equation}
where $u$ denotes the cell density and $v$ stands for the concentration of the chemical signal emitted by cells.  $\gamma(v)>0$ is the diffusion coefficient and $\chi(v)$ is called the chemotactic coefficient, both of them depend on the chemical signal concentration and satisfy  the following proportionality relation:
\begin{equation}\label{KS-1}
\phi(v)=(\alpha-1)\gamma'(v),
\end{equation}
where $\alpha$ denotes the ratio of effective body length (i.e. maximal distance between receptors) to cell step size. {\color{black} We refer the detailed derivation of \eqref{KS}-\eqref{KS-1} to \cite{KS-1971-JTB2}.} The prominent feature of the Keller-Segel system \eqref{KS} is that two coefficient $\gamma(v)$  and  $\phi(v)$ depend on the chemical signal concentration and have a prescribed relationship to each other.  {\color{black} Recently this proportionality relation with $\alpha=0$ and $\gamma'(v)<0$ has been advocated as ``density-suppressed motility mechanism'' to interpret the stripe pattern formation of engineered {\it Escherichia Coli} in \cite{Fu-PRL-2011, Liu-Science}, which will be elaborated later. Such signal-dependent motility mechanism has also been used in preytaxis to describe the spatially inhomogeneous distribution of coexistence in the predator-prey system (see \cite{Kareiva, JWEJAM}). There are some other chemotaxis models where the diffusive and chemotactic coefficients depend on the chemical concentration gradient (cf. \cite{Burni-Chouhad}) or cell density (cf. \cite{Wrzosek}), which clearly have different modeling point of view from the system \eqref{KS}-\eqref{KS-1}}.

The study of Keller-Segel system \eqref{KS} was started with simplified cases. If $\gamma(v)=1$ and $\phi(v)=\chi>0$ ($\chi$ is a constant), the system \eqref{KS} is simplified to the so-called minimal Keller-Segel (abbreviated as KS) model:
\begin{equation}\label{CKS}
\begin{cases}
u_t=\Delta u-\chi \nabla \cdot(u\nabla v),&x\in\Omega,t>0,\\
v_t=\Delta v+u-v, &x\in\Omega,t>0,\\
\end{cases}
\end{equation}
where $\Omega$ is a bounded domain in $\R^n$ with smooth boundary. Under homogeneous Neumann boundary conditions, the dynamics of \eqref{CKS} such as boundedness, blow-up and pattern formation have been extensively studied, see the review papers \cite{BBTW-M3AS-2015,HP-JMB-2009,H-Review-1,H-Review-2} for more details. The most prominent phenomenon is the existence of critical mass depending on the space dimensions. Precisely, the global bounded solutions exists in one dimension \cite{KA-FE-2001}. In space of two dimensions ($n=2$), there exists a critical mass $m_*=\frac{4\pi}{\chi}$ such that the solution is bounded and asymptotically converges to its unique constant equilibrium if $\int_\Omega u_0 dx<m_*$ \cite{Nagai-Funk, WWY} and blows up if $\int_\Omega u_0 dx>m_*$ \cite{horstmann2001blow}, where $\int_\Omega u_0 dx$ denotes the initial cell mass. In the higher dimensions ($n\geq 3$), for any $\int_\Omega u_0dx>0$, the solution may blow up in finite time \cite{W-JMPA-2013}. The mathematical analysis for the KS model \eqref{CKS} on the boundedness {\it vs.} blowup was essentially based on the following Lyapunov functional
\begin{equation}\label{3-2}
\begin{split}
F(u,v)=&\int_{\Omega}u\ln udx +\frac{\chi}{2}\int_\Omega ( v^2+\abs{\nabla v}^2)dx-\chi\int_{\Omega}uvdx.
\end{split}
\end{equation}
{\color{black} If $\gamma(v)=1$ and $\phi(v)=\frac{\chi}{v}$, the system \eqref{KS} becomes the so-called singular Keller-Segel system:
\begin{equation}\label{CKSN}
\begin{cases}
u_t=\Delta u-\chi \nabla \cdot(\frac{u}{v}\nabla v),&x\in\Omega,t>0,\\
v_t=\Delta v+u-v, &x\in\Omega,t>0,\\
\end{cases}
\end{equation}
and there are various results in the literature indicating the non-existence of blow-up of solutions. With homogeneous Neumann boundary conditions, the existence of globally bounded solutions of \eqref{CKSN} was established if $n=2$ and $\chi<\chi_0$ for some $\chi_0>1$ \cite{La-M2AS-2016} or $n\geq 3$ and $\chi<\sqrt{\frac{2}{n}}$ \cite{F-JMAA-2015,W-M2AS-2011}. Moreover, when $\chi<\sqrt{\frac{2}{n}}$ and $n\geq 2$, the asymptotic stability of constant steady states was obtained in \cite{WY-NATM-2018}. More results on the radially symmetric case or weak solutions can be found in \cite{SW-NA-2011,FS-Nonlinearity-2016, B-AMSA-1998,FS-DCDSB-2016,FWY-M2AS-2015} and we refer to \cite{BBTW-M3AS-2015} for more details.
}

Turning to the full KS system \eqref{KS} where $\gamma(v)$ and $\phi(v)$ are nonconstant functions satisfying \eqref{KS-1}, to our knowledge, the known results are only limited to the special case $\phi(v)=-\gamma'(v)$ (i.e. $\alpha=0$ in \eqref{KS-1}) which simplifies the KS system \eqref{KS} into
\begin{equation}\label{KS-2}
\begin{cases}
u_t=\Delta(\gamma(v)u),&x\in\Omega,\ t>0,\\
 v_t=\Delta v+u-v,&x\in\Omega,\ t>0.
\end{cases}
\end{equation}
Here the parameter $\alpha=0$ in \eqref{KS-1} means that ``the distance between receptors is zero and the chemotaxis occurs because of an undirected effect on activity due to the presence of a chemical sensed by a single receptor" as stated in \cite[On page 228]{KS-1971-JTB2}. Recently to describe the stripe pattern formation observed in the experiment of \cite{Liu-Science},  a so-called density-suppressed motility model was proposed in \cite{Fu-PRL-2011} as follows
\begin{equation}\label{KSN}
\begin{cases}
u_t=\Delta(\gamma(v)u)+\sigma u(1-u),&x\in\Omega,\ t>0,\\
 v_t=\Delta v+u-v,&x\in\Omega,\ t>0.
\end{cases}
\end{equation}
with $\gamma'(v)<0$ and $\sigma\geq 0$ denotes the intrinsic cell growth rate. Clearly the density-suppressed motility model \eqref{KSN} with $\sigma=0$ coincides with the simplified KS model \eqref{KS-2}.

When the homogeneous Neumann boundary conditions are imposed, there are some results available to \eqref{KS-2} and \eqref{KSN}. First for the system \eqref{KS-2}, it was shown that globally bounded solutions exist in two dimensional spaces by assuming that the motility function $\gamma(v)\in C^3{([0,\infty)}\cap W^{1,\infty}(0,\infty))$ has both positive lower and upper bounds \cite{TW-M3AS-2017}. It turns out that the uniformly positive assumption on $\gamma(v)$ (i.e. $\gamma(v)$ has a positive lower bound) is not necessary to ensure the global boundedness of solutions. For example, if $\gamma(v)=\frac{\chi}{v^k}$ (i.e. $\gamma(v)$ decays algebraically), it has been proved that global bounded solutions exist in all dimensions provided $\chi>0$ is small enough \cite{YK-AAM-2017} or in two dimensional spaces for parabolic-elliptic simplification of the system \eqref{KS-2} (see \cite{AY-Nonlinearity-2019}). For the system \eqref{KSN} with $\gamma'(v)<0$,  it was shown that global bounded solutions exist in two dimensions for any $\sigma>0$ \cite{JKW-SIAP-2018} and in higher dimensions $(n\geq 3)$ for large $\sigma>0$ \cite{WW}. The results of \cite{JKW-SIAP-2018} essentially rely on the assumption $\sigma>0$. Therefore a natural question is whether the solution of \eqref{KSN} with $\sigma=0$ (i.e. KS system \eqref{KS-2} with $\gamma'(v)<0$) is globally bounded ? This question has been partially confirmed in \cite{YK-AAM-2017, AY-Nonlinearity-2019} for algebraically decay function $\gamma(v)$ with various conditions as mentioned above.  The purpose of this paper is to investigate the same question for exponentially decay motility function $\gamma(v)= e^{-\chi v}$ with $\chi>0$. That is we consider the following problem
\begin{equation}\label{1-1}
\begin{cases}
u_t=\Delta (e^{-\chi v}u), &x\in \Omega, ~~t>0,\\
 v_t=\Delta v+u- v,& x\in \Omega, ~~t>0,\\
 \frac{\partial u}{\partial\nu}=\frac{\partial v}{\partial \nu}=0, &x\in\partial\Omega,t>0,\\
u(x,0)=u_0(x),~v(x,0)=v_0(x), & x\in \Omega,
\end{cases}
\end{equation}
where $\Omega\subset\R^n$ is a bounded domain with smooth boundary and $\nu$ stands for the outward unit normal vector on $\partial\Omega$. Surprisingly, we find that uniform-in-time boundedness of solutions of \eqref{1-1} is no longer true and the solution may blow-up in two dimensions, which is quite different from the results of \cite{YK-AAM-2017, AY-Nonlinearity-2019} for algebraically decay function $\gamma(v)$. Our result indicates that the solution behavior of the system \eqref{KS-2} may essentially depend on the decay rate of the motility function $\gamma(v)$. The main results of this paper are the following.
{\color{black}\begin{theorem}\label{BS}
Let $\Omega\subset \R^2 $ be a  bounded domain with smooth boundary. Assume that $0\leq(u_0,v_0)\in [W^{1,\infty}(\Omega)]^2$. Then the following results hold true.

(i) If $\int_\Omega u_0dx<\frac{4\pi}{\chi}$,  then the system (\ref{1-1})  admits a unique classical solution $(u,v)\in [C^0(\bar{\Omega}\times[0,\infty))\cap C^{2,1}(\bar{\Omega}\times(0,\infty))]^2$ satisfying
\begin{equation*}
\|u(\cdot,t)\|_{L^\infty}+\|v(\cdot,t)\|_{W^{1,\infty}}\leq C,
\end{equation*}
where $C$ is a constant independent of $t$.

(ii) For any $M>\frac{4\pi}{\chi}$ and $M\not \in \{\frac{4\pi m}{\chi}: m\in\mathbb{N}^+\}$ where $\mathbb{N}^+$ denotes the set of positive integers, there exist initial data $(u_0,v_0)$ satisfying $\int_\Omega u_0dx=M$ such that the corresponding solution blows up in finite/infinite time.

%If $M>\frac{4\pi}{\chi}$ and $M\not \in \{\frac{4\pi m}{\chi}: m\in\mathbb{N}^+\}$ where $\mathbb{N}^+$ denotes the set of positive integers, then there exist initial data such that the solutions of  $\eqref{1-1}$ blow up in finite/infinite time.
\end{theorem}}

We remark that the blowup result in Theorem \ref{BS}$(ii)$ does not assert the finiteness or infiniteness of blowup time, which leaves out an interesting question for the future study. {\color{black} Moreover, as we know the nonlinear diffusion may play an important role in blow-up dynamics such as blow-up rate (see \cite{KPN-MA-2017,BP-CPAA-2012} and references therein). Hence it would be of interest to study qualitative properties of blow-up solutions to the system \eqref{1-1} in the future.}

{\color{black} The new contribution of this paper lies in the finding of the critical mass phenomenon for the system \eqref{KS-2} with exponentially  decay motility function $\gamma(v)$. This new finding along with the existing results in \cite{YK-AAM-2017, AY-Nonlinearity-2019} for \eqref{KS-2} with algebraically decay function $\gamma(v)$ shows that the dynamics of \eqref{KS-2} is very rich and complex where the decay rate of the motility function $\gamma(v)$ will play a key role. This provides us a heuristic direction to {\color{black}further} explore the dynamics of the full Keller-Segel system \eqref{KS} whose dynamics has been only partially understood so far for the special case $\alpha=0$ in \eqref{KS-1}, namely for \eqref{KS-2}. Technically to overcome the possible degeneracy, we develop the weighted energy estimates by treating the degenerate term as a weight function to achieve the results in Theorem \ref{BS}. This technique may become a common (if not necessary) tool to study chemotaxis systems with the signal-dependent degenerate diffusion}.

One can check that the system \eqref{1-1} has the same Lyapunov functional \eqref{3-2} as for the minimal KS model \eqref{CKS}, which can be used to construct some initial data with large negative energy such that the solution of \eqref{1-1} blows up  for supercritical mass (i.e.,$\int_\Omega u_0dx>\frac{4\pi}{\chi}$).  Moreover, under the subcritical mass (i.e. $\int_\Omega u_0dx<\frac{4\pi}{\chi}$), using the same Lyapunov functional and  Trudinger-Moser inequality,  we can find a constant $c_1>0$ such that
\begin{equation}\label{KS-4}
\|u\ln u\|_{L^1}+\|\nabla v\|_{L^2}+\int_0^t\| v_t\|_{L^2}^2ds\leq c_1
\end{equation}
which has been a key to prove the boundedness of solutions of the minimal KS system \eqref{CKS}. However, there are some significant differences between systems \eqref{CKS} and \eqref{1-1}. For the minimal KS model \eqref{CKS}, the estimate \eqref{KS-4} is enough to establish the existence of global classical solutions (see \cite{Nagai-Funk}). However for the system \eqref{1-1}, the motility coefficient $e^{-\chi v}$ may touch down to zero (degenerate) as $v\to\infty$, and hence the method for the constant diffusion as in \cite{Nagai-Funk} no longer works and new ideas are demanded. In this paper, we shall develop the weighted energy estimates by taking $e^{-\chi v}$ as the weight function based on the Lyapunov functional to establish our results.

\section{Local existence and basic inequalities}
Using Amann's theorem \cite{A-DIE-1990, A-Book-1993} (cf. also \cite[Lemma 2.6]{Wang-Hillen}) or the well-established fixed point argument together with the parabolic regularity theory \cite{TW-2015-SIMA,JKW-SIAP-2018}, we can show the existence and uniqueness of local solutions of $\eqref{1-1}$. We omit the details of the proof for brevity.
\begin{lemma}\label{LS}Let $\Omega\subset \R^2 $ be a  bounded domain with smooth boundary.
Assume that $0\leq(u_0,v_0)\in [W^{1,\infty}(\Omega)]^2$. Then there exists $T_{max}\in(0,\infty]$ such that the problem \eqref{1-1} has a unique classical solution $(u,v)\in[C(\bar{\Omega}\times[0,T_{max}))\cap C^{2,1}(\bar{\Omega}\times(0,T_{max}))]^2$. Moreover $u,v>0$ in $\Omega\times(0,T_{max})$ and
\begin{equation*}\label{BC}
if ~T_{max}<\infty, ~~then~~\norm{u(\cdot,t)}_{L^\infty}+\|v(\cdot,t)\|_{W^{1,\infty}}\to\infty ~~ as~~~t\nearrow T_{max}.
\end{equation*}
\end{lemma}

\begin{lemma} If $(u,v)$ is a  solution  of $\eqref{1-1}$ in $\Omega\times(0,T)$ for some $T>0$, then
\begin{equation}\label{L1-u}
\|u(\cdot,t)\|_{L^1}=\|u_0\|_{L^1}:\equiv M_0, \ \mathrm{for\ all}\ t\in(0,T)
\end{equation}
and
\begin{equation}\label{L1-v}
\|v(\cdot,t)\|_{L^1}\leq \|u_0\|_{L^1}+\|v_0\|_{L^1},  \ \mathrm{for\ all}\ t\in(0,T).
\end{equation}
\end{lemma}
\begin{proof}Integrating the first equation of  \eqref{1-1} and using the Neumann boundary conditions, we obtain \eqref{L1-u} directly. On the other hand, integrating the second equation of \eqref{1-1} with respect to $x$ over $\Omega$, one has
\begin{equation*}
\frac{d}{dt}\int_\Omega vdx+\int_\Omega vdx=\int_\Omega udx=\int_\Omega u_0dx,
\end{equation*}
which immediately gives \eqref{L1-v}.
\end{proof}
{\color{black}{
\begin{lemma}\label{kb*}
Let $\Omega\subset\R^2$ be a bounded domain with smooth boundary. Assume $\mathcal{A}$ is a self-adjoint realization of $-\Delta$  defined on $D(\mathcal{A}):=\{\psi\in W^{2,2}(\Omega)\cap L^{2}(\Omega)|\int_{\Omega}\psi=0\text{~~and~~}\frac{\partial\psi}{\partial\nu}=0~~\text{on}~~\partial\Omega\}$.
Then for any $L>0$ and a nonnegative function $f$ satisfying
\begin{equation}\label{lne}
\int_\Omega f\ln f dx\leq L,
\end{equation}
it holds that
\begin{equation}\label{kb}
\int_\Omega |\mathcal{A}^{-\frac{1}{2}}(f-\bar{f})|^2dx\leq C(L),
\end{equation}
where $\bar{f}=\frac{1}{|\Omega|}\int_\Omega f dx$.
\end{lemma}
\begin{proof}
Using \eqref{lne} and noting the fact $z\ln z\geq -\frac{1}{e}$ for all $z>0$, we have
\begin{equation*}
\begin{split}
\|f\|_{L^1}
&=\int_{f\geq e}fdx +\int_{f<e} fdx\\
&\leq \int_{f\geq e} f\ln f dx+\int_{f< e} fdx=\int_\Omega f\ln f-\int_{f<e}f\ln fdx +\int_{f<e}fdx\\
&\leq L+\frac{|\Omega|}{e}+e|\Omega|,
\end{split}
\end{equation*}
and hence
\begin{equation}\label{kb-2}
\|f-\bar{f}\|_{L^1}\leq 2\|f\|_{L^1}\leq 2L+\frac{2|\Omega|}{e}+2e|\Omega|.
\end{equation}
Next, we consider the following system
\begin{equation}\label{kb-1}
\begin{cases}
-\Delta \phi=f-\bar{f},&x\in\Omega,\\
\frac{\partial \phi}{\partial \nu}=0,&x\in\partial\Omega.
\end{cases}
\end{equation}
Let $G$ denote the Green's function of $-\Delta$ in $\Omega$ with the homogeneous Neumann boundary condition. From \eqref{kb-1}, one has
\begin{equation}\label{kb-2*}
\phi(x)=\int_\Omega G(x-y)(f(y)-\bar{f})dy.
\end{equation}
Then using the similar argument as in \cite[Lemma A.3]{TW-JDE-2014} along with \eqref{kb-2}, from \eqref{kb-2*} one can find a constant $\kappa>0$ such that
\begin{equation}\label{kb-3}
\int_\Omega e^{\kappa |\phi|}dx\leq c_1.
\end{equation}
Recall a result (see \cite[Lemma A.2]{TW-JDE-2014}): for $\kappa>0$, it holds
\begin{equation*}
XY\leq \frac{1}{\kappa}X\ln X+\frac{1}{\kappa e}e^{\kappa Y}\ \mathrm{for\ all}\ \ X>0 \ \mathrm{and}\ \  Y>0.
\end{equation*}
Then multiplying the first equation of $\eqref{kb-1}$ by $\phi$, and integrating it by parts, we end up with
\begin{equation}\label{kb-4}
\begin{split}
\int_\Omega |\nabla \phi|^2dx
=\int_\Omega f\phi-\bar{f}\int_\Omega \phi dx
&\leq \int_\Omega f|\phi|+\int_\Omega \bar{f} |\phi| dx\\
%&\leq \frac{1}{\kappa}\int_\Omega f\ln fdx+\frac{2}{\kappa e}\int_\Omega e^{\kappa |\phi|}dx+\frac{1}{\kappa}\int_\Omega \bar{f}\ln \bar{f}\\
&\leq\frac{1}{\kappa}\int_\Omega f\ln fdx+\frac{2}{\kappa e}\int_\Omega e^{\kappa |\phi|}dx+\frac{|\Omega|}{\kappa} \bar{f}\ln \bar{f}.
\end{split}
\end{equation}
Substituting \eqref{lne} and \eqref{kb-3} into \eqref{kb-4}, and using the boundedness of $\bar{f}\ln \bar{f}$, one has
\begin{equation}\label{kb-5}
\int_\Omega |\nabla \phi|^2dx\leq c_2
\end{equation}
where $c_2$ depends on $L$.
The definition of $\mathcal{A}$ defines the self-adjoint fractional powers $\mathcal{A}^{-\delta}$ for any $\delta>0$. Then from \eqref{kb-1} we have $\phi=\mathcal{A}^{-1}(f-\bar{f})$ and hence
\begin{equation*}
\int_\Omega |\mathcal{A}^{-\frac{1}{2}}(f-\bar{f})|^2dx=\int_\Omega \mathcal{A}^{-1}(f-\bar{f})(f-\bar{f})dx=\int_\Omega \phi (-\Delta \phi)dx=\int_\Omega |\nabla \phi|^2dx\leq c_2,
\end{equation*}
which gives \eqref{kb}.
\end{proof}
}}

\begin{lemma} \label{2L1}
Let $(u,v)$ be a solution of the system (\ref{1-1}). Then there exists a constant $C>0$ independent of $t$ such that
\begin{equation}\label{2Lnm}
\|\Delta v\|_{L^2}\leq C(\|u\|_{L^2}+\|v_t\|_{L^2}) .
\end{equation}
\end{lemma}
\begin{proof} Noting that $v$ satisfies the following system
\begin{equation}\label{EE}
\begin{cases}
-\Delta v+v=u-v_t,&x\in\Omega,\ \ t>0,\\
\frac{\partial v}{\partial\nu}=0,&x\in\Omega,\ \ t>0.\\
\end{cases}
\end{equation}
Then applying the Agmon-Douglis-Nirenberg $L^p$ estimates(see \cite{ADN-1959,ADN-1964}) to the system \eqref{EE}, we can find a constant $c_1>0$ such that
\begin{equation*}
\|v\|_{W^{2,2}}\leq c_1\|(u-v_t)\|_{L^2}\leq 2c_1(\|u\|_{L^2}+\|v_t\|_{L^2}),
\end{equation*}
which gives \eqref{2Lnm}.
\end{proof}

%\begin{lemma}[Trudinger-Moser inequality \cite{Nagai-Funk}]
%Let $\Omega$ be a bounded domain in $\R^2$ with smooth boundary. Then  for any $\varepsilon>0$,   there exists a constant $C_\varepsilon$ depending on $\varepsilon$ and $\Omega$ such that
%\begin{equation}\label{T-M}
%\int_\Omega \exp{|u|}dx\leq C_\varepsilon \exp\left\{\left(\frac{1}{8\pi}+\varepsilon\right)\|\nabla u\|_{L^2}^2+\frac{1}{|\Omega|}\|u\|_{L^1}\right\}.
%\end{equation}
%\end{lemma}
%\begin{lemma}[\cite{Nagai-Funk}] Let $\Omega$ be a bounded domain in $\R^2$ with smooth boundary. Then for any $\varepsilon>0$, there exists a positive constant $C_\varepsilon$ such that
%\begin{equation}\label{L3 estimate}
%\norm{w}_{L^3}\leq \varepsilon \norm{\nabla w}_{L^2}^\frac{2}{3}\norm{w\ln w}_{L^1}^\frac{1}{3}+C_\varepsilon(\|w\ln w\|_{L^1}+\|w\|_{L^1}^\frac{1}{3}).
%\end{equation}
%\end{lemma}

\section{Proof of Theorem \ref{BS}}

In this section, we shall prove Theorem \ref{BS}, which includes the global existence of classical solutions for subcritical mass and blowup of solutions for supercritical mass.
\begin{lemma}\label{L-Y}
Let $F(u,v)$ be defined in $\eqref{3-2}$. Then the solutions of $\eqref{1-1}$ satisfy
\begin{equation}\label{3-3}
\frac{d}{dt}F(u,v)+{\color{black}E}(u,v)=0,
\end{equation}
where
\begin{equation*}\label{3-4}
E(u,v)=\chi \int_\Omega v_t^2 dx+\int_{\Omega}e^{-\chi v}u|\nabla(\ln u-\chi v)|^2dx.
\end{equation*}
\end{lemma}
\begin{proof}
We multiply the first equation of $\eqref{1-1}$ by $(\ln u-\chi v)$ and integrate the result with respect to $x$ over $\Omega$ to have
\begin{equation}\label{3-5}
\begin{split}
\int_\Omega u_t(\ln u-\chi v )dx
&=\int_{\Omega}\nabla\cdot(e^{-\chi v}\nabla u-\chi e^{-\chi v} u\nabla v)(\ln u-\chi v)dx\\
&=-\int_{\Omega}e^{-\chi v}u|\nabla(\ln u-\chi v)|^2dx.
\end{split}
\end{equation}
On the other hand, using the fact that $\int_\Omega u_tdx=0$, we have
\begin{equation}\label{3-6}
\begin{split}
\int_\Omega u_t(\ln u-\chi v)dx=\frac{d}{dt}\int_\Omega u\ln udx-\chi \frac{d}{dt}\int_\Omega uv dx+\chi\int_{\Omega} uv_tdx.
\end{split}
\end{equation}
From the second equation of $\eqref{1-1}$, one has
$
u=v_t-\Delta v+v,
$
which gives
\begin{equation}\label{3-8}
\int_{\Omega} uv_tdx=\int_\Omega v_t^2 dx+\frac{1}{2}\frac{d}{dt}\int_\Omega |\nabla v|^2 dx+\frac{1}{2}\frac{d}{dt}\int_\Omega v^2 dx.
\end{equation}
Then the combination of \eqref{3-5}, \eqref{3-6} and \eqref{3-8} gives \eqref{3-3}.
\end{proof}
\subsection{Global existence with subcritical mass}
In this subsection, we first prove the existence of global classical solutions if $\int_\Omega u_0dx<\frac{4\pi}{\chi}$.
\begin{lemma}\label{ln-e}
If  $\int_\Omega u_0dx<\frac{4\pi}{\chi}$, then there exists a constant $C>0$ independent of $t$
{\color{black}such that
\begin{equation}\label{ln}
\int_\Omega u\ln udx\leq C
\end{equation}
and
\begin{equation}\label{ln-1}
\|\nabla v(\cdot,t)\|_{L^2}^2+\int_0^t \|v_t(\cdot,s)\|_{L^2}^2ds\leq C.
\end{equation}

}

%and
%\begin{equation}\label{4-1}
%\|\nabla v(\cdot,t)\|_{L^2}^2\leq C
%\end{equation}
%as well as
%\begin{equation}\label{vt}
%\int_0^t \|v_t(\cdot,s)\|_{L^2}^2ds\leq C.
%\end{equation}
\end{lemma}
\begin{proof}
 From $\eqref{3-2}$, we have that
\begin{equation}\label{4-2}
\begin{split}
F(u,v)&=\int_{\Omega}u\ln udx -(\chi+\eta) \int_{\Omega}uvdx+\frac{\chi}{2}\int_\Omega ( v^2+\abs{\nabla v}^2)dx+\eta\int_{\Omega}uvdx\\
     &=-\int_\Omega u\ln \frac{e^{\left(\chi+\eta\right )v}}{u}dx+\frac{\chi}{2}\int_\Omega ( v^2+\abs{\nabla v}^2)dx+\eta\int_{\Omega}uvdx\end{split}
\end{equation}
 Noting that $-\ln z$ is a convex function for all $z\geq 0$ and $\int_\Omega \frac{u}{M_0}dx=1$, which allows us to use the Jensen's inequality to obtain
 \begin{equation}\label{4-5}
 \begin{split}
 -\ln\left(\frac{1}{M_0}\int_\Omega e^{(\chi+\eta)v}dx\right)
 &=-\ln\left(\int_\Omega\frac{ e^{(\chi+\eta)v}}{u}\frac{u}{M_0}dx\right)\\
 &\leq\int_\Omega\left(-\ln\frac{ e^{(\chi+\eta)v}}{u} \right)\frac{u}{M_0}dx=-\frac{1}{M_0}\int_\Omega u\ln\frac{e^{(\chi+\eta)v}}{u} dx.
 \end{split}
 \end{equation}
Then the combination of $\eqref{4-2}$ and $\eqref{4-5}$ gives
 \begin{equation}\label{4-6}
  \begin{split}
  F(u,v)&\geq -M_0\ln\left(\frac{1}{M_0}\int_\Omega e^{(\chi+\eta)v}dx\right)+\frac{\chi}{2}\int_\Omega ( v^2+\abs{\nabla v}^2)dx+\eta\int_{\Omega}uvdx.
  \end{split}
 \end{equation}
Noting the fact $\|v\|_{L^1}\leq c_1$, and using the Trudinger-Moser inequality in two dimensional spaces \cite{Nagai-Funk}, one has
 \begin{equation}\label{4-7}
 \begin{split}
 \int_\Omega e^{\left(\chi+\eta\right )v}dx
 &\leq c_2
 e^{\left(\frac{1}{8\pi}+\varepsilon\right)\left(\chi+\eta\right )^2\|\nabla v\|_{L^2}^2},\\
 \end{split}
 \end{equation}
 which substituted  into $\eqref{4-6}$ gives
 \begin{equation}\label{4-8}
 \begin{split}
 F(u,v)&\geq\left[\frac{\chi}{2}-\left(\frac{1}{8\pi}+\varepsilon\right)\left(\chi+\eta\right )^2M_0\right]\int_\Omega|\nabla v|^2dx+\frac{\chi}{2}\int_\Omega v^2dx+\eta\int_\Omega uvdx-c_3,
 \end{split}
 \end{equation}
where $c_3:=M_0\ln  \frac{c_2}{M_0}$. Since $M_0=\int_\Omega u_0 dx<\frac{4\pi}{\chi}$, it holds that
\begin{equation}\label{4-8*}
\frac{\chi}{2}-\left(\frac{1}{8\pi}+\varepsilon\right)\left(\chi+\eta\right )^2M_0>0,
\end{equation}
by choosing  $\varepsilon>0$ and $\eta>0$ small enough. Substituting $\eqref{4-8*}$ into $\eqref{4-8}$, one has
 \begin{equation*}\label{4-9}
 \begin{split}
 F(u,v)\geq &\frac{\chi}{2}\int_\Omega v^2dx+\eta\int_\Omega uvdx-c_3,\\
  \end{split}
 \end{equation*}
which gives $F(u,v)\geq -c_3$ and $\int_\Omega uvdx\leq \frac{F(u_0,v_0)+c_3}{\eta}$ by  the fact $F(u,v)\leq F(u_0,v_0)$.
Then using the definition of $F(u,v)$ in $\eqref{3-2}$ and  the fact $F(u,v)\leq F(u_0,v_0)$ again, we obtain
\begin{equation*}\label{i1}
\int_\Omega u\ln u dx\leq F(u,v)+\chi \int_\Omega uvdx \leq \left(1+\frac{\chi}{\eta}\right) F(u_0,v_0)+\frac{\chi c_3}{\eta},
\end{equation*}
{\color{black}which gives \eqref{ln}}.
Moreover, we have the following estimate
 \begin{equation}\label{i2}
 \begin{split}
\frac{\chi}{2}\int_\Omega |\nabla v|^2 dx &\leq F(u,v)+\chi \int_\Omega uvdx-\int_\Omega u\ln u dx\\%-\frac{\chi}{2\alpha}\int_\Omega\left(|\nabla v|^2+\beta v^2\right)dx-\frac{\xi}{2\gamma}\int_{\Omega}(\delta w^2+|\nabla w|^2)dx\\
 &\leq F(u,v)+\chi \int_\Omega uv dx+\frac{|\Omega|}{e}\leq \left(1+\frac{\chi}{\eta}\right) F(u_0,v_0)+\frac{\chi c_3}{\eta}+\frac{|\Omega|}{e}.
 \end{split}
 \end{equation}
Integrating \eqref{3-3} and using the fact $F(u,v)\geq -c_3$, it follows that
 \begin{equation*}
\chi \int_0^t \int_\Omega v_t^2 dxdt+\int_0^t\int_{\Omega}e^{-\chi v}u|\nabla(\ln u-\chi v)|^2dxdt\leq F(u_0,v_0)-F(u,v)\leq F(u_0,v_0)+c_3,
 \end{equation*}
 which yields
 \begin{equation}\label{i3}
 \int_0^t \int_\Omega v_t^2 dxdt\leq \frac{F(u_0,v_0)+c_3}{\chi}.
 \end{equation}
Thus the combination of \eqref{i2}-\eqref{i3} {\color{black} gives \eqref{ln-1}} and  completes the proof.

\end{proof}

\begin{lemma}\label{kb-kb}
Let $(u,v)$ be a solution of  (\ref{1-1}). If  $\int_\Omega u_0(x)dx<\frac{4\pi}{\chi}$, then there exists a constant $C>0$ independent of $t$ such that the following inequality holds
{\color{black}\begin{equation}\label{Lr2}
\int_t^{t+\tau}\int_\Omega e^{-\chi v}u^2dxds\leq C,\ \
\ \mathrm{for\ all}\  t\in (0,\widetilde{T}_{max}).
\end{equation}}
where
\begin{equation}\label{tau}
\tau:=\min\{1,\frac{1}{2}T_{max}\}\ \mathrm{and}\ \widetilde{T}_{max}=\begin{cases}
T_{max}-\tau \ &\mathrm{if}\ T_{max}<\infty,\\
\infty \ &\mathrm{if}\ T_{max}=\infty.
\end{cases}
\end{equation}
\end{lemma}
\begin{proof} Using the definition of $\mathcal{A}$ in Lemma \ref{kb*}, we can rewrite the system \eqref{1-1} as follows
\begin{equation}\label{AE}
\begin{cases}
(u-\bar{u})_t=-\mathcal{A} (e^{-\chi v}u-\overline{e^{-\chi v}u}),&x\in\Omega, t>0,\\
\frac{\partial u}{\partial \nu}=0, &x\in\partial \Omega, t>0,\\
\end{cases}
\end{equation}
Then multiplying \eqref{AE} by $\mathcal{A}^{-1}(u-\bar{u})$ and integrating the result by parts, we have
\begin{equation}\label{0-3}
\begin{split}
\frac{1}{2}\frac{d}{dt}\int_{\Omega}|\mathcal{A}^{-\frac{1}{2}}(u-\bar{u})|^{2}dx
&=-\int_{\Omega}\mathcal{A}^{-1}\left(u-\bar{u}\right)\cdot\mathcal{A}\left(e^{-\chi v}u-\overline{e^{-\chi v}u}\right)dx\\
&=-\int_{\Omega}(u-\bar{u})\cdot\left(e^{-\chi v}u-\overline{e^{-\chi v}u}\right)dx.\\
\end{split}
\end{equation}
On the other hand, with some direct calculations and noting that $\bar{u}=\frac{1}{|\Omega|}\int_\Omega udx=\frac{M_0}{|\Omega|}$, it holds
\begin{equation}\label{0-4}
\begin{split}
-\int_{\Omega}(u-\bar{u})\cdot\left(e^{-\chi v}u-\overline{e^{-\chi v}u}\right)dx
&=-\int_{\Omega}(u-\bar{u})\Big(e^{-\chi v}(u-\bar{u})+e^{-\chi v}\bar{u}-\overline{e^{-\chi v}u}\Big)dx
\\
&= -\int_{\Omega}e^{-\chi v}(u-\bar{u})^{2}dx+\bar{u}\int_\Omega (\bar{u}-u)e^{-\chi v}dx\\
&\leq -\int_{\Omega}e^{-\chi v}(u-\bar{u})^{2}dx+\frac{M_0^2}{|\Omega|}.
\end{split}
\end{equation}
Then we substitute \eqref{0-4} into \eqref{0-3} to get
\begin{equation}\label{0-5}
\frac{d}{dt}\int_{\Omega} |\mathcal{A}^{-\frac{1}{2}}(u-\bar{u})|^2dx+2\int_{\Omega}e^{-\chi v}(u-\bar{u})^{2}dx\leq \frac{2M_0^2}{|\Omega|}.
\end{equation}
{\color{black}{Since $\int_\Omega u_0dx<\frac{4\pi}{\chi}$, then from Lemma \ref{ln-e} and Lemma \ref{kb*}, we can find a constant $c_1>0$ such that
\begin{equation}\label{0-4*}
\int_{\Omega} |\mathcal{A}^{-\frac{1}{2}}(u-\bar{u})|^2dx\leq c_1.
\end{equation}
}}
Then integrating \eqref{0-5} over $(t,t+\tau)$ and using \eqref{0-4*}, one has
\begin{equation*}
\begin{split}
\int_t^{t+\tau}\int_\Omega e^{-\chi v}(u-\bar{u})^2 dxds\leq \frac{M_0^2}{|\Omega|}\tau\leq \frac{M_0^2}{|\Omega|} ,
\end{split}
\end{equation*}
which gives
\begin{equation*}
\begin{split}
\int_t^{t+\tau}\int_\Omega e^{-\chi v}u^2dxds
&=\int_{t}^{t+\tau}\int_\Omega e^{-\chi v}(u-\bar{u}+\bar{u})^2dxds\\
&\leq 2\int_{t}^{t+\tau}\int_\Omega e^{-\chi v}(u-\bar{u})^2dxds+2\int_{t}^{t+\tau}\int_\Omega \bar{u}^2dxds\leq \frac{4M_0^2}{|\Omega|},
\end{split}
\end{equation*}
and hence \eqref{Lr2} follows. Then we complete the proof.
\end{proof}

\begin{lemma}
Suppose the conditions in Lemma \ref{kb-kb} hold. Then there exists a constant $C>0$ independent of $t$ such that
{\color{black}\begin{equation}\label{L2i}
\int_t^{t+\tau}\|v(\cdot,s)\|_{L^\infty}ds\leq C,\ \
\ \mathrm{for\ all}\  t\in (0,\widetilde{T}_{max}).
\end{equation}}
where $\tau$ is defined by \eqref{tau}.
\end{lemma}
\begin{proof}
Using the Sobolev embedding theorem and applying the Agmon-Douglis-Nirenberg $L^p$ estimates(see \cite{ADN-1959,ADN-1964}) to the system \eqref{EE}, we have
\begin{equation}\label{4-7}
\begin{split}
\|v\|_{L^\infty}
\leq c_1
\|v\|_{W^{2,\frac{3}{2}}}
&\leq c_2\|(u-v_t)\|_{L^\frac{3}{2}}\\
&\leq 2c_2(\|u\|_{L^\frac{3}{2}}+\|v_t\|_{L^\frac{3}{2}})\\
&\leq 2c_2\left(\int_\Omega u^2 e^{-\chi v}dx\right)^\frac{1}{2}\cdot\left(\int_\Omega e^{3\chi v}dx\right)^\frac{1}{6}+2c_2\left(\int_\Omega v_t^2dx\right)^\frac{1}{2}|\Omega|^\frac{1}{6}\\
&\leq c_2^2\int_\Omega u^2 e^{-\chi v}dx+\left(\int_\Omega e^{3\chi v}dx\right)^\frac{1}{3}+c_2^2\|v_t\|_{L^2}^2+|\Omega|^\frac{1}{3}.
\end{split}
\end{equation}
On the other hand, using the fact $\|v\|_{L^1}+\|\nabla v\|_{L^2}\leq c_3$ (see Lemma \ref{ln-e} and Lemma \ref{kb*})  and applying  the Trudinger-Moser inequality in two dimensional spaces \cite{Nagai-Funk}, one has
$
 \int_\Omega e^{3\chi v}dx\leq c_4,\\
$
 which, substituted into \eqref{4-7} and combined with \eqref{Lr2} and \eqref{ln-1}, gives
 \begin{equation*}
 \begin{split}
 \int_t^{t+\tau}\|v(\cdot,s)\|_{L^\infty}ds
 &\leq c_2^2\int_t^{t+\tau}\int_\Omega u^2 e^{-\chi v}dx+c_2^2\int_t^{t+\tau}\|v_t(\cdot,s)\|_{L^2}^2ds+c_4\leq c_5,
 \end{split}
 \end{equation*}
 which yields \eqref{L2i}.
\end{proof}

With the above results in hand, we shall show that  there exists a constant $C>0$ such that $\|u(\cdot,t)\|_{L^2}\leq C$ for any $t\in(0,T_{max})$, which will be used to rule out the possibility of degeneracy. Precisely, we have the following results.
\begin{lemma}\label{2L2}
 Let $\Omega\subset\R^2$ be a bounded domain with smooth boundary and $\int_\Omega u_0dx<\frac{4\pi}{\chi}$. If $(u,v)$ is a solution  of system (\ref{1-1}) in $\Omega\times(0,T_{max})$, then there exists a positive constant $C$ independent of $t$ such that
\begin{equation}\label{2L2-1}
\|u(\cdot,t)\|_{L^2}\leq C , \ \ \mathrm{for\ all}\ t\in(0,T_{max}).
\end{equation}
\end{lemma}
\begin{proof}
We multiply the first equation of  (\ref{1-1}) by $u$ and integrate the result by parts with respect to $x$.  Then using the H\"{o}lder inequality and Young's inequality, we have
\begin{equation*}\label{2L2-2*}
\begin{split}
\frac{1}{2}\frac{d}{dt}\int_\Omega u^2dx+\int_\Omega e^{-\chi v}|\nabla u|^2dx&=\chi\int_\Omega e^{-\chi v} u\nabla u\cdot \nabla vdx\\
&\leq \chi\left(\int_\Omega e^{-\chi v} |\nabla u|^2dx\right)^\frac{1}{2}\left(\int_\Omega e^{-\chi v} u^2|\nabla v|^2dx \right)^\frac{1}{2} \\
&\leq \frac{1}{2}\int_\Omega e^{-\chi v}|\nabla u|^2dx+\frac{\chi^2}{2}\int_\Omega e^{-\chi v} u^2|\nabla v|^2dx,
\end{split}
\end{equation*}
%where we have  used the Young's inequality:
%\begin{equation}
%\label{YI}AB\leq \varepsilon A^p+(\varepsilon p)^{-\frac{q}{p}}q^{-1}B^q \ \ \mathrm{for\ any}\ \ \varepsilon >0, \ \ p,q>0,\ \ \frac{1}{p}+\frac{1}{q}=1.
%\end{equation}
which yields
\begin{equation}\label{2L2-2}
\frac{d}{dt}\int_\Omega u^2dx+\int_\Omega e^{-\chi v}|\nabla u|^2dx\leq \chi^2\int_\Omega e^{-\chi v} u^2|\nabla v|^2dx.
\end{equation}
 On the other hand,  using the fact $|X+Y|^2\geq {\color{black}\frac{1}{2}X^2}-Y^2$ and
$
e^{-\frac{\chi}{2}v}\nabla u= \nabla (e^{-\frac{\chi}{2}v} u)+\frac{\chi}{2}e^{-\frac{\chi}{2}v} u\nabla v,
$
 we have
\begin{equation*}\label{2L2-3}
e^{-\chi v} |\nabla u|^2
\geq {\color{black}\frac{1}{2}|\nabla (e^{-\frac{\chi}{2}v} u)|^2}-\frac{\chi^2}{4} e^{-\chi v} u^2|\nabla v|^2,
\end{equation*}
which substituted  into \eqref{2L2-2} gives
\begin{equation}\label{2L2-4}
\frac{d}{dt}\int_\Omega u^2dx+\frac{1}{2}\int_\Omega  |\nabla (e^{-\frac{\chi}{2}v} u)|^2dx
\leq \frac{5\chi^2}{4}\int_\Omega e^{-\chi v}  u^2|\nabla v|^2dx\\
\leq  \frac{5\chi^2}{4}\|\nabla v\|_{L^4}^2\|e^{-\frac{\chi}{2}v}u\|_{L^4}^2.
\end{equation}
Moreover, the Gagliardo-Nirenberg inequality along with the facts $\|\nabla v\|_{L^2}\leq c_1$ and $\|\nabla v\|_{L^4}\leq c_2(\|\Delta v\|_{L^2}^\frac{1}{2}\|\nabla v\|_{L^2}^\frac{1}{2}+\|\nabla v\|_{L^2})$ (see \cite[Lemma 2.5]{JKW-SIAP-2018}) entails  that
\begin{equation*}\label{2L2-6}
\begin{split}
&\frac{5\chi^2}{4}\|\nabla v\|_{L^4}^2\|e^{-\frac{\chi}{2}v}u\|_{L^4}^2\\
&\leq c_3(\|\Delta v\|_{L^2}\|\nabla v\|_{L^2}+\|\nabla v\|_{L^2}^2)(\|\nabla (e^{-\frac{\chi}{2} v} u)\|_{L^2}\|e^{-\frac{\chi}{2} v} u\|_{L^2}+\|e^{-\frac{\chi}{2} v} u\|_{L^2}^2)\\
%&\leq c_1c_3(\|\Delta v\|_{L^2}+c_1)(\|\nabla (e^{-\frac{\chi}{2} v} u)\|_{L^2}\|e^{-\frac{\chi}{2} v} u\|_{L^2}+\|e^{-\frac{\chi}{2} v} u\|_{L^2}^2)\\
&\leq c_1c_3\|\Delta v\|_{L^2}\|\nabla(e^{-\frac{\chi}{2}v}u)\|_{L^2}\|e^{-\frac{\chi}{2}v}u\|_{L^2}+c_1c_3\|\Delta v\|_{L^2}\|e^{-\frac{\chi}{2}v}u\|_{L^2}^2\\
&\ \ \ \ +c_1^2c_3\|\nabla (e^{-\frac{\chi}{2} v} u)\|_{L^2}\|e^{-\frac{\chi}{2} v} u\|_{L^2}+c_1^2c_3\|e^{-\frac{\chi}{2}v}u\|_{L^2}^2\\
&\leq \frac{1}{2}\|\nabla (e^{-\frac{\chi}{2} v} u)\|_{L^2}^2+2c_1^2c_3^2\|\Delta v\|_{L^2}^2\|e^{-\frac{\chi}{2}v}u\|_{L^2}^2+\frac{1+4c_1^4c_3^2}{4}\|e^{-\frac{\chi}{2}v}u\|_{L^2}^2\\
&\leq \frac{1}{2}\|\nabla (e^{-\frac{\chi}{2} v} u)\|_{L^2}^2+c_4(\|\Delta v\|_{L^2}^2+1)\|e^{-\frac{\chi}{2} v}u\|_{L^2}^2,\\
\end{split}
\end{equation*}
which, combined with \eqref{2Lnm} and the fact $e^{-\chi v}\leq 1$, gives
\begin{equation}\label{2L2-6}
\begin{split}
\frac{5\chi^2}{4}\|\nabla v\|_{L^4}^2\|e^{-\frac{\chi}{2}v}u\|_{L^4}^2
&\leq \frac{1}{2}\|\nabla (e^{-\frac{\chi}{2} v} u)\|_{L^2}^2+c_5\left(\|u\|_{L^2}^2+\|v_t\|^2_{L^2}+1\right)\|e^{-\frac{\chi}{2} v}u\|_{L^2}^2\\
&\leq \frac{1}{2}\|\nabla (e^{-\frac{\chi}{2} v} u)\|_{L^2}^2+ c_5\left(\|e^{-\frac{\chi}{2} v}u\|_{L^2}^2+\|v_t\|^2_{L^2}+1\right)
\|u\|_{L^2}^2.
\end{split}
\end{equation}
Substituting \eqref{2L2-6} into \eqref{2L2-4}, one has
\begin{equation}\label{2L2-6*}
\frac{d}{dt}\|u\|_{L^2}^2\leq c_5\left(\|e^{-\frac{\chi}{2} v}u\|_{L^2}^2+\|v_t\|^2_{L^2}+1\right)
\|u\|_{L^2}^2.
\end{equation}
For any $t\in (0,T_{max})$ and in the case of either $t\in(0,\tau)$ or $t\geq \tau$ with $\tau=\min \Big\{ 1,\frac{1}{2} T_{max}\Big\}$, from \eqref{Lr2} we can find a $t_0=t_0(t)\in ((t-\tau)_+, t)$ such that $t_0\geq 0$ and
$
\int_\Omega e^{-\chi v(x,t_0)}u^2(x,t_0)dx\leq c_6,
$
which, along with $\eqref{L2i}$, implies that
\begin{equation}\label{2L2-k}
\int_\Omega u^2(x,t_0)dx\leq c_7.
\end{equation}
Integrating \eqref{2L2-6*} over $(t_0,t)$ and noting the fact $t\leq t_0+\tau\leq t_0+1$, then we can  use \eqref{ln-1}, \eqref{Lr2} and \eqref{2L2-k} to obtain
\begin{equation*}
\begin{split}
\|u(\cdot,t)\|_{L^2}^2\leq \|u(\cdot,t_0)\|_{L^2}^2 e^{c_5\int_{t_0}^t\|e^{-\frac{\chi}{2} v}u\|_{L^2}^2ds+c_5\int_{t_0}^t \|v_t\|_{L^2}^2ds+c_5 \tau}\leq c_8\|u_0\|_{L^2}^2 ,
\end{split}
\end{equation*}
which gives \eqref{2L2-1}.
\end{proof}

\begin{lemma}\label{boundedness-1}Let the conditions in Lemma \ref{2L2} hold. Suppose $(u,v)$ is a solution of  \eqref{1-1} in $\Omega\times(0,T_{max})$, then one has
\begin{equation}\label{LI-u}
\|u(\cdot,t)\|_{L^\infty}\leq C, \ \mathrm{for\ all}\ \ t\in(0,T_{max}),
\end{equation}
where $C>0$ is a constant independent of $t$.
\end{lemma}
\begin{proof}
%By the variation-of-constant formula, it follows from the second equation of  \eqref{1-1} that
%\begin{equation}\label{Lv-1}
%v(x,t)=e^{(\Delta -1)t}v_0(x)+\int_0^t e^{(\Delta-1)(t-s)}u(x,s)ds.
%\end{equation}
%Then applying the well-known smoothing estimates for the  heat semigroup $(e^{\Delta t})_{t\geq 0}$ on homogeneous Neumann boundary conditions and using \eqref{2L2-1}, one has from \eqref{Lv-1} that
%\begin{equation}\label{Lv-3}
%\begin{split}
%\|v(\cdot,t)\|_{L^\infty}
%&\leq e^{-t}\|v_0\|_{L^\infty}+c_1\int_0^t(t-s)^{-\frac{1}{2}}e^{-(t-s)}\|u(\cdot,s)\|_{L^2}ds\\
%&\leq \|v_0\|_{L^\infty}+c_1c_2 \int_0^t(t-s)^{-\frac{1}{2}}e^{-(t-s)}ds\leq c_3
%\end{split}
%\end{equation}
%and
%\begin{equation}\label{Lv-4}
%\begin{split}
%\|\nabla v\|_{L^4}
%&\leq \|\nabla v_0\|_{L^4}+c_4 \int_0^t(t-s)^{-\frac{3}{4}}e^{-(t-s)}ds\leq c_5.
%\end{split}
%\end{equation}
{\color{black}Noting \eqref{2L2-1} and applying the parabolic regularity estimates to the second equation of \eqref{1-1}, one can find a positive constant $c_1$ such that
\begin{equation}\label{Lv-4}
\|v\|_{L^\infty}+\|\nabla v\|_{L^4}\leq c_1,
\end{equation}
which gives
\begin{equation}\label{Lv-5}
e^{-\chi v}\geq e^{-\chi c_1}:=d_1>0.
\end{equation}}
Then multiplying the first equation of \eqref{1-1} by $u^{p-1}$ with $p\geq 3$ and integrating the result by parts, we end up with
\begin{equation*}
\begin{split}
&\frac{1}{p}\frac{d}{dt}\int_\Omega u^pdx+(p-1)\int_\Omega e^{-\chi v}u^{p-2}|\nabla u|^2dx\\
&=-(p-1)\chi\int_\Omega e^{-\chi v}u^{p-1}\nabla u\cdot\nabla vdx\\
&\leq \frac{p-1}{2}\int_\Omega e^{-\chi v}u^{p-2}|\nabla u|^2dx+\frac{(p-1)\chi^2}{2}\int_\Omega e^{-\chi v}u^p|\nabla v|^2dx,
\end{split}
\end{equation*}
which, combined with \eqref{Lv-5} and the fact $e^{-\chi v}\leq 1$, gives
\begin{equation}\label{Lv-6}
\frac{d}{dt}\int_\Omega u^p dx+\frac{2(p-1)d_1}{p}\int_\Omega |\nabla u^\frac{p}{2}|^2dx\leq \frac{p(p-1)\chi^2}{2}\int_\Omega u^p|\nabla v|^2dx.
\end{equation}
Then with \eqref{2L2-1} and \eqref{Lv-4}, we can use the H\"{o}lder's inequality and Gagliardo-Nirenberg inequality to get
\begin{equation}\label{Lv-7}
\begin{split}
\frac{p(p-1)\chi^2}{2}\int_\Omega u^p|\nabla v|^2dx
&\leq \frac{p(p-1)\chi^2}{2}\left(\int_\Omega u^{2p}dx\right)^\frac{1}{2}\cdot\left(\int_\Omega |\nabla v|^4dx\right)^\frac{1}{2}\\
&\leq \frac{p(p-1)\chi^2}{2}\|u^\frac{p}{2}\|_{L^4}^2\|\nabla v\|_{L^4}^2\\
&\leq \frac{p(p-1)\chi^2}{2} c_2(\|\nabla u^{\frac{p}{2}}\|_{L^2}^{2(1-\frac{1}{p})}\|u^\frac{p}{2}\|_{L^\frac{4}{p}}^\frac{2}{p}
+\|u^\frac{p}{2}\|_{L^\frac{4}{p}}^2)\\
%&\leq d_2\|\nabla u^{\frac{p}{2}}\|_{L^2}^{2(1-\frac{1}{p})}+d_3\\
&\leq  \frac{(p-1)d_1}{p}\|\nabla u^\frac{p}{2}\|_{L^2}^2+c_3,
%\frac{d_1}{p}\left(\frac{d_2}{d_1}\right)^p+d_3,
\end{split}
\end{equation}
%where
%\begin{equation*}
%d_2:=\frac{p(p-1)\chi^2c_2c_5^2c_6 }{2} \ \ \mathrm{and}\ \ d_3:=\frac{p(p-1)\chi^2c_2^pc_5^2c_6 }{2}.
%\end{equation*}
On the other hand, using the Gagliardo-Nirenberg inequality and \eqref{2L2-1} again, one has
\begin{equation}\label{Lv-7*}
\begin{split}
\int_\Omega u^p dx=\|u^\frac{p}{2}\|_{L^2}^2
&\leq c_4(\|\nabla u^\frac{p}{2}\|_{L^2}^{2(1-\frac{2}{p})}\|u^\frac{p}{2}\|_{L^\frac{4}{p}}^\frac{4}{p}+\|u^\frac{p}{2}\|_{L^\frac{4}{p}}^2)\\
%&\leq c_7c_2^2\|\nabla u^\frac{p}{2}\|_{L^2}^{2(1-\frac{2}{p})}+c_7c_2^p\\
&\leq \frac{(p-1)d_1}{p}\|\nabla u^\frac{p}{2}\|_{L^2}^2+c_5.
\end{split}
\end{equation}
%and
%\begin{equation*}
%d_3:=\frac{p(p-1)\chi^2c_2^pc_5^2c_6 }{2} .
%\end{equation*}
%Then using the Young's inequality, one can derive that
%\begin{equation}
%d_2(T)\|\nabla u^{\frac{p}{2}}\|_{L^2}^{2(1-\frac{1}{p})}\leq \frac{(p-1)d_1(T)}{p}\|\nabla u^\frac{p}{2}\|_{L^2}^2+\frac{2d_1(T)}{p}\left(\frac{d_2(T)}{d_1(T)}\right)^p
%\end{equation}
%Using the Gagliardo-Nirenberg inequality and Young's inequality again, and noting that $\|u^\frac{p}{2}\|_{L^\frac{2}{p}}=\|u\|_{L^1}^frac{p}{2}=M^\frac{p}{2}$,  we obtain
%\begin{equation}\label{Lv-8}
%\begin{split}
%\int_\Omega u^p dx=\|u^\frac{p}{2}\|_{L^2}^2
%&\leq c_8(\|\nabla u^\frac{p}{2}\|_{L^2}^{2(1-\frac{1}{p})}\|u^\frac{p}{2}\|_{L^\frac{2}{p}}^\frac{2}{p}
%+\|u^\frac{p}{2}\|_{L^\frac{2}{p}}^2)\\
%&\leq c_8M \|\nabla u^\frac{p}{2}\|_{L^2}^{2(1-\frac{1}{p})}+c_8M^p\\
%&\leq \frac{(p-1)d_1(T)}{p}\|\nabla u^\frac{p}{2}\|_{L^2}^2+\frac{d_1(T)}{p}\left(\frac{c_8M}{d_1(T)}\right)^p+c_8M^p.
%\end{split}
%\end{equation}
Then substituting \eqref{Lv-7} and \eqref{Lv-7*} into \eqref{Lv-6}, and integrating the result with respect to $t$, we have for all {\color{black}$t\in(0,T_{max})$} that
\begin{equation}\label{Lv-8}
\|u(\cdot,t)\|_{L^p}^p\leq \|u_0\|_{L^p}^p+c_6,
\end{equation}
where $c_6>0$ is constant depend on $p$ but independent of $t$.
Applying the parabolic regularity theory to the second equation of \eqref{1-1}, and choosing $p=4$ in \eqref{Lv-8}, one can find a positive constant $d_2$ independent of $p$ such that $\|\nabla v(\cdot,t)\|_{L^\infty}\leq d_2$. Then by the well-known Moser iteration \cite{Alikakos-1979}(or see \cite{JKW-SIAP-2018}), we can show \eqref{LI-u}.
\end{proof}

\begin{lemma}\label{BS-1}
Let $\Omega\subset \R^2 $ be a  bounded domain with smooth boundary. Assume that $0\leq(u_0,v_0)\in [W^{1,\infty}(\Omega)]^2$ and then if $\int_{\Omega}u_0dx<\frac{4\pi}{\chi}$, the system (\ref{1-1}) admits a unique classical solution $(u,v)\in [C^0(\bar{\Omega}\times[0,\infty))\cap C^{2,1}(\bar{\Omega}\times(0,\infty))]^2$ with uniform-in-time bound.
\end{lemma}
\begin{proof}Using Lemma \ref{boundedness-1}, we have
$
\|u(\cdot,t)\|_{L^\infty}\leq C_1.
$
Then applying the parabolic regularity to the second equation of  \eqref{1-1}, one has
$\|v(\cdot,t)\|_{W^{1,\infty}}\leq C_2$. Hence Lemma \ref{BS-1} follows directly by using Lemma \ref{LS}.
\end{proof}

\subsection{Blowup for supercritical mass}
In this subsection, we shall construct some initial data {\color{black}with supercritical mass (i.e., $\int_\Omega u_0 dx>\frac{4\pi}{\chi}$)} such that the corresponding solution of  \eqref{1-1} blows up based on some ideas in \cite{horstmann2001blow,JW-JDE-2016}.
%The proof is  based on some ideas in \cite{horstmann2001blow,JW-JDE-2016} by using the Lyapunov functional $F(u,v)$ defined in \eqref{3-2}.
%Precisely, the proof consists of the following steps:
%\begin{itemize}
% \item[Step 1:] If $\int_\Omega u_0(x)dx\neq \frac{4\pi m}{\chi}$  for any $m\in\mathbb{N^+}$, then we find a constant $K>0$ such that $F(u_\infty,v_\infty)\geq-K$, where $(u_\infty,v_\infty)$ is  the steady state of  \eqref{1-1}.\\
%\item[Step 2:] If $\int_\Omega u_0(x)dx> \frac{4\pi }{\chi}$, we  construct some initial data $(u_0,v_0)\in [W^{1,\infty}(\Omega)]^2$ such that $F(u_0,v_0)<-K$.\\
%\item[Step 3:] By supposing $(u,v)$ is a global and bounded solution of  \eqref{1-1},  we can find a time sequence $t_k\to\infty$ such that $(u(\cdot,t_k),v(\cdot,t_k)) \to (u_\infty,v_\infty)$ satisfying
%$F(u_\infty,v_\infty)\leq F(u_0,v_0)$
%which raises the following contradiction:
%$
%   -K\leq F(u_\infty,v_\infty)\leq F(u_0,v_0)<-K.
%$
%\end{itemize}
 Noting $M_0=\int_\Omega u_0 dx$, then the stationary solution of system \eqref{1-1} satisfies the following problem
\begin{equation}\label{ST-1}
\begin{cases}
-\Delta v+ v=\frac{M_0 e^{\chi v}}{\int_\Omega e^{\chi v}dx }, \ \ &x\in\Omega,\\
u=\frac{M_0 e^{\chi v}}{\int_\Omega e^{\chi v}dx},\ \ &x\in\Omega,\\
\frac{\partial v}{\partial \nu}=0, &x\in\partial\Omega,\\
\int_\Omega v dx=\int_\Omega udx=M_0.
\end{cases}
\end{equation}
%$\lambda>0$ is a constant satisfying
%\begin{equation}\label{5-6}
%\lambda=\frac{\int_\Omega udx}{\int_\Omega e^{\chi v}dx}=\frac{M}{\int_\Omega e^{\chi v}dx}.
%\end{equation}
%with $u=\frac{M_0 e^{\chi v}}{\int_\Omega e^{\chi v}dx}$.
For convenience, we introduce the following change of variable:
$
V=v-\frac{1}{|\Omega|}\int_\Omega v dx=v-\frac{M_0}{|\Omega|}.
$
Then the system \eqref{ST-1} can be rewritten as
\begin{equation}\label{s-1}
\begin{cases}
-\Delta V+ V=\frac{M_0 e^{\chi V}}{\int_\Omega e^{\chi V}dx }-\frac{ M_0}{|\Omega|},&x\in\Omega,\\
U=\frac{M_0 e^{\chi V}}{\int_\Omega e^{\chi V}dx }, &x\in\Omega,\\
\frac{\partial V}{\partial \nu}=0,&x\in \partial\Omega, \\
\int_\Omega Vdx=0,\ \ \int_\Omega U dx=M_0.
\end{cases}
\end{equation}
We  point out that the steady state problem \eqref{s-1} and the Lyapunov function \eqref{3-2} for \eqref{1-1}  are the same as those for the minimal Keller-Segel system \eqref{CKS} whose blow-up of solutions has been studied in \cite{horstmann2001blow,horstmann2001nonsymmetric}. Hence we  use the same arguments  as in \cite[Lemma 3.5]{horstmann2001blow} to establish the lower bound for the steady-state energy
  when $\int_\Omega u_0dx\neq \frac{4\pi m}{\chi}$  for any $m\in\mathbb{N^+}$. For convenience, we cite the results without proof.
\begin{lemma}\label{S-B}
Suppose $M_0\neq \frac{4\pi m}{\chi}$ for all $m\in\mathbb{N^+}$. Then there exists a constant $K> 0$ such that
  \begin{equation}\label{s*}
F(U,V)\geq- K
  \end{equation}
holds for any solution  $(U,V)$ of the system $\eqref{s-1}$.
\end{lemma}
Next, we show that there exist  some initial data with supercritical mass (i.e., $M_0>\frac{4\pi}{\chi}$) such that the energy is below any prescribed bound. To this end, we first prove that there is a sequence $(U_\varepsilon,V_\varepsilon)_{\varepsilon> 0}$ satisfying $\int_\Omega V_\varepsilon (x)dx=0$ and $\int_\Omega U_\varepsilon(x) dx=M_0$ such that $\lim\limits_{\varepsilon\to 0}F(U_\varepsilon,V_\varepsilon)=-\infty$ if $M_0>\frac{4\pi}{\chi}$.
 %From \cite[p. 615]{chen1991classification}, we know that the functions
%\begin{equation*}
%\psi_\varepsilon(x)=\ln\left(\frac{8\varepsilon^2}{(\varepsilon^2+\pi|x-x_0|^2)^2}\right),\ \varepsilon>0,  x_0\in\R^2
%\end{equation*}
%are solutions of
%$
%-\Delta \psi(x)=e^{\psi(x)},~~x\in\R^2$ satisfying $\int_{\R^2}e^{\psi(x)}dx <\infty.$
%We note that as $\varepsilon\to 0$, $\psi_\varepsilon(x)\to-\infty$ for all $x\neq x_0$ and $\psi_\varepsilon(x_0)\to\infty$.
Let $(U_\varepsilon,V_\varepsilon)$ be defined as follows:
\begin{equation}\label{2-4-12}
\begin{split}
V_\varepsilon(x)
&=\frac{1}{\chi}\left[\ln\left(\frac{\varepsilon^2}{(\varepsilon^2+\pi|x-x_0|^2)^2}\right)-\frac{1}{|\Omega|}\int_\Omega\ln\left(\frac{\varepsilon^2}{(\varepsilon^2+\pi|x-x_0|^2)^2}\right)dx\right],
\end{split}
\end{equation}
and
\begin{equation}\label{2-4-12*}
U_\varepsilon(x)=\frac{M_0e^{\chi V_\varepsilon(x)}}{\int_\Omega e^{\chi V_\varepsilon(x)}dx},
\end{equation}
where $x_0$ is an arbitrary point on $\partial\Omega$.
One can easily check that $\int_\Omega V_\varepsilon (x)dx=0$ and $\int_\Omega U_\varepsilon(x) dx=M_0$. Next, we shall show that $\lim\limits_{\varepsilon\to 0}F(U_\varepsilon,V_\varepsilon)=-\infty$ if $M_0>\frac{4\pi}{\chi}$.
\begin{lemma}\label{K-B}
Let $(U_\varepsilon, V_\varepsilon)_{\varepsilon> 0}$ be defined by $\eqref{2-4-12}- \eqref{2-4-12*}$ and $x_0\in\partial\Omega$. If $M_0>\frac{4\pi}{\chi}$, then it holds that
\begin{equation}\label{5-2-3}
F(U_\varepsilon,V_\varepsilon)\to-\infty \ \ \mathrm{as}\ \ \varepsilon\to 0.
\end{equation}
\end{lemma}
\begin{proof}
Since $x_0$ is an arbitrary point on $\partial\Omega$, we assume $x_0=0$ without loss of generality. With the definition of $F(u,v)$ and \eqref{2-4-12*}, one has
\begin{equation}\label{5-2-4}
\begin{split}
F(U_\varepsilon,V_\varepsilon)
&=\int_{\Omega}U_\varepsilon\ln U_\varepsilon dx-\chi \int_{\Omega}U_\varepsilon V_\varepsilon dx+\frac{\chi }{2}\int_\Omega \abs{\nabla V_\varepsilon}^2dx+\frac{\chi}{2}\int_\Omega V_\varepsilon^2dx\\
&=M_0\ln M_0-M_0\ln\left(\int_\Omega  e^{\chi V_\varepsilon}dx\right)+\frac{\chi }{2}\int_\Omega \abs{\nabla V_\varepsilon}^2dx+\frac{\chi}{2}\int_\Omega V_\varepsilon^2dx,
\end{split}
\end{equation}
where we have used the fact
\begin{equation*}\label{5-2-5}
\begin{split}
&\int_{\Omega}U_\varepsilon\ln U_\varepsilon dx-\chi\int_{\Omega}U_\varepsilon V_\varepsilon dx\\
&=\frac{M_0}{\int_\Omega e^{\chi V_\varepsilon}dx}\int_\Omega e^{\chi V_\varepsilon}\left[\ln M_0+\chi V_\varepsilon-\ln\left(\int_\Omega e^{\chi V_\varepsilon}dx\right)\right]dx-\frac{\chi M_0}{\int_\Omega e^{\chi V_\varepsilon}dx}\int_\Omega e^{\chi V_\varepsilon}V_\varepsilon dx\\
&=M_0\ln M_0-M_0\ln\left(\int_\Omega  e^{\chi V_\varepsilon}dx\right).
\end{split}
\end{equation*}
On the other hand, we use $\eqref{2-4-12}$ and the polar coordinates around origin
$0\in\partial \Omega$, with $R$ denoting  the maximum distance between the pole and boundary of $\Omega$, to derive that
\begin{equation}\label{2-4-15}
\begin{split}
\frac{\chi }{2}\int_\Omega \abs{\nabla V_\varepsilon}^2dx
&\leq \frac{8\pi^2}{\chi}\int_0^{\pi}\int_0^{\frac{R}{\varepsilon}}\frac{r^3}{(1+\pi r^2)^2}drd\theta\\
&\leq\frac{4\pi}{\chi}\left(\ln \frac{1}{\varepsilon^2}+\ln (\varepsilon^2+\pi R^2)-1+\frac{\varepsilon^2}{\varepsilon^2+\pi R^2}\right)\\
&\leq\frac{8\pi}{\chi}\ln\frac{1}{\varepsilon}+O_1(1),
\end{split}
\end{equation}
where $|O_1(1)|\leq C$ as $\varepsilon\to0$. Moreover, direct calculations give
\begin{equation}\label{2-4-18}
\begin{split}
\frac{\chi}{2}\int_\Omega V_\varepsilon^2dx
&=\frac{1}{2\chi}\int_\Omega (\ln(\varepsilon^2+\pi|x|^2)^2)^2dx-\frac{1}{2\chi |\Omega|}\left(\int_\Omega \ln(\varepsilon^2+\pi|x|^2)^2dx\right)^2=O_2(1),
\end{split}
\end{equation}
where $|O_2(1)|\leq C$ as $\varepsilon\to0$. Furthermore, it has that
%\begin{equation*}
%\begin{split}
%\int_\Omega e^{\chi v_\varepsilon}dx
%=|\Omega| e^{-\frac{1}{|\Omega|}
%\int_\Omega\ln\left(\frac{\varepsilon^2}{(\varepsilon^2+\pi|x|^2)^2}\right)dx}\int_\Omega\left(\frac{\varepsilon^2}{(\varepsilon^2+\pi|x|^2)^2}\right)dx
%\end{split}
%\end{equation*}
%and
\begin{equation}\label{ve*}
\begin{split}
\ln\left(\int_\Omega  e^{\chi V_\varepsilon}dx\right)=\ln\left(|\Omega|\int_\Omega\frac{\varepsilon^2}{(\varepsilon^2+\pi|x|^2)^2}dx\right)
-\frac{1}{|\Omega|}\int_\Omega\ln\left(\frac{\varepsilon^2}{(\varepsilon^2+\pi|x|^2)^2}\right)dx,
\end{split}
\end{equation}
and
$$1-\frac{\varepsilon^2}{\pi R_1^2+\varepsilon^2}\leq \int_\Omega\frac{\varepsilon^2}{(\varepsilon^2+\pi|x|^2)^2}dx\leq 1-\frac{\varepsilon^2}{\pi R_2^2+\varepsilon^2},$$
  where $R_1$ and $R_2$ denote the maximum and minimum distance between the pole and the boundary of  $\Omega$. Then from \eqref{ve*}, one can show that
\begin{equation}\label{ve}
\begin{split}
&-M_0\ln \left(\int_\Omega e^{\chi V_\varepsilon}dx\right)\\
&=-M_0\left[\ln\left(|\Omega|\int_\Omega\frac{\varepsilon^2}{(\varepsilon^2+\pi|x|^2)^2}dx\right)
-\frac{1}{|\Omega|}\int_\Omega\ln\left(\frac{\varepsilon^2}{(\varepsilon^2+\pi|x|^2)^2}\right)dx\right]\\
&= \frac{M_0}{|\Omega|}\int_\Omega \ln \varepsilon^2dx+\frac{M_0}{|\Omega|}\int_\Omega \ln(\varepsilon^2+\pi|x|^2)^2dx-M_0\ln\left(|\Omega|\int_\Omega\frac{\varepsilon^2}{(\varepsilon^2+\pi|x|^2)^2}dx\right)\\
&=2M_0\ln\varepsilon+O_3(1),\\
\end{split}
\end{equation}
with $|O_3(1)|\leq C $ as $\varepsilon\to0$. Finally substituting \eqref{2-4-15}, \eqref{2-4-18} and \eqref{ve} into \eqref{5-2-4} gives
\begin{equation}\label{2-4-20}
F(U_\varepsilon,V_\varepsilon)\leq2\left(\frac{4\pi}{\chi}- M_0\right)\ln\frac{1}{ \varepsilon} +O(1),
\end{equation}
where $O(1)=O_1(1)+O_2(1)+O_3(1)$ and $|O(1)|\leq C$ as $\varepsilon \to 0$. Since $M_0>\frac{4\pi}{\chi}$, \eqref{5-2-3} follows directly from \eqref{2-4-20}.
\end{proof}
 Next, we shall establish the connection between the energy of steady states and the initial data. More precisely, we have the following results.
\begin{lemma}\label{S-P}
Let  $(u,v)$ be a  global-in-time bounded solution of $\eqref{1-1}$. Then  there exist a sequence of times $t_k\to\infty$ and nonnegative function $(U_\infty,V_\infty)\in [C^2(\bar{\Omega})]^2$ such that $(u(\cdot,t_k),v(\cdot,t_k)) \to (U_\infty,V_\infty)$ in $[C^2(\bar{\Omega})]^2$. Furthermore,  $(U_\infty,V_\infty)$ is a solution  of (\ref{s-1}) satisfying
\begin{equation}\label{s-2}
F(U_\infty,V_\infty)\leq F(u_0,v_0).
\end{equation}
\end{lemma}
\begin{proof}
 Since $(u,v)$ is the global classical solution with uniform-in-time bound of the system \eqref{1-1}, {\color{black}then we can use the standard bootstrap arguments involving interior parabolic regularity theory \cite{Lady-1968} to find a constant $c_1>0$ independent of $t$  such that
 \begin{equation}\label{he}
 \|u(\cdot,t)\|_{C^{2+\sigma,1+\frac{\sigma}{2}}(\bar{\Omega}\times[1,\infty))}
 +\|v(\cdot,t)\|_{C^{2+\sigma,1+\frac{\sigma}{2}}(\bar{\Omega}\times[1,\infty))}\leq c_1,
 \end{equation}
 where $\sigma\in(0,1)$. From \eqref{he}, we know that $(u(\cdot,t),v(\cdot,t))_{t>1}$ is relatively compact in $[C^2(\bar{\Omega})]^2$ and $F(u,v)$ is bounded for $t>1$. Hence there exists a suitable time sequence  $t_k\to\infty$ such that$(u(\cdot,t_k),v(\cdot,t_k)) \to (U_\infty,V_\infty)$ in $[C^2(\bar{\Omega})]^2$ for some nonnegative $U_\infty,V_\infty\in C^2(\bar{\Omega})$. Then we have $$F(u(\cdot,t_k),v(\cdot,t_k))\to F(U_\infty,V_\infty), \ \mathrm{as} \ t_k\to\infty,$$
which gives \eqref{s-2} by the fact $F(u,v)\leq F(u_0,v_0)$ from \eqref{3-3}. On the other hand, using the facts $0<c_2\leq e^{-\chi v}$ and $F(u,v)$ is bounded for $t>1$, from Lemma $\ref{L-Y}$, one has
\begin{equation}\label{he-1}
\int_1^\infty\int_{\Omega}v_t^2dxds
+\int_1^\infty\int_{\Omega}u|\nabla(\ln u-\chi v)|^2dxds\leq c_3.
\end{equation}
Then the combination of \eqref{he} and \eqref{he-1} entails us to extract a subsequence of $(t_k)_{k\geq 1}$ (with the same notation if necessary)  such that
\begin{equation}\label{he-2}
\int_\Omega v_t^2(x,t_k)dx \to 0
\ \ \mathrm{as}\ \ t_k\to\infty
\end{equation}
and
\begin{equation}\label{he-3}
\int_{\Omega}u(x,t_k)|\nabla(\ln u(\cdot,t_k)-\chi v(\cdot,t_k))|^2dx \to 0
\ \ \mathrm{as}\ \ t_k\to\infty.
\end{equation}
Based on \eqref{he-2} and \eqref{he-3}, then using the same argument as in \cite[Lemma 3.1]{W-2010}, we can show that $(U_\infty,V_\infty)$ is a solution  of (\ref{s-1}).
}
%we can use  Schauder regularity theory (e.g. see \cite{Gilbard-Ttrudinger}), to show that $(u(\cdot,t),v(\cdot,t))_{t>1}$ is relatively compact in $[C^2(\bar{\Omega})]^2$. Hence there exists a suitable time sequence  $(t_k)_{k\geq 1}$ such that $(u(\cdot,t_k),v(\cdot,t_k)) \to (U_\infty,V_\infty)$ in $[C^2(\bar{\Omega})]^2$ as $t_k\to\infty$.
 %Moreover, using the boundedness of $(u,v)$, we can find a constant $c_1>0$ such that
 %\begin{equation*}
 %F(u,v)
 %=\int_{\Omega}u\ln udx +\frac{\chi}{2}\int_\Omega ( v^2+\abs{\nabla v}^2)dx-\chi\int_{\Omega}uvdx\\
 %\geq -c_1,
 %\end{equation*}
%which combined with Lemma $\ref{L-Y}$ implies that
%\begin{equation*}\label{s-3}
%\chi\int_0^\infty\int_{\Omega}v_t^2dxd\tau
%+\int_0^\infty\int_{\Omega}e^{-\chi v}u|\nabla(\ln u-\chi v)|^2dxds<\infty.
%\end{equation*}
%Then applying the  Arzel$\grave{a}$-Ascoli theorem,  we can find a time sequence, still denoted by $(t_k)_{k\geq 1}$, such that
%\begin{equation*}\label{s-4}
%v_t(\cdot,t_k)\to 0  ~~\mathrm{in}~L^2(\Omega)
%\end{equation*}
% and
% \begin{equation}\label{s-5}
% e^{-\chi v(\cdot,t_{k})}u(\cdot,t_k)|\nabla(\ln u(\cdot,t_k)-\chi v(\cdot,t_k))|^2\to 0 ~~\mathrm{a.e.~in}~~\bar{\Omega}
% \end{equation}
%
In fact, noting \eqref{he-2}, we evaluate the second equation of $\eqref{1-1}$ at $t=t_k$ and let $k\to\infty$ to have
 \begin{equation}\label{s-6}
 -\Delta V_\infty+ V_\infty=U_\infty-\bar{u}.
 \end{equation}
 Using $\eqref{he-3}$ and taking $k\to\infty$, we obtain
$
U_\infty|\nabla(\ln U_{\infty}-\chi V_{\infty})|^2=0~\mathrm{in}~\bar{\Omega}.
$
By using the  same argument as in \cite[Lemma 3.1]{W-2010}, one can show that $U_{\infty}>0$ for all $x\in \bar{\Omega}$ and hence
$
\nabla(\ln U_{\infty}-\chi V_{\infty})=0~\mathrm{in}~\bar{\Omega}
$
which gives
\begin{equation}\label{s-7}
U_{\infty}=\frac{M_0 e^{\chi V_{\infty}}}{\int_\Omega e^{\chi V_{\infty}}dx} .
\end{equation}
Then combining $\eqref{s-6}$ and $\eqref{s-7}$, and using the fact $\bar{u}=\frac{M_0}{|\Omega|}$, we know that $(U_\infty,V_\infty)$ is a solution of \eqref{s-1}. %Furthermore, because $(u(\cdot,t_k),v(\cdot,t_k)) \to (U_\infty,V_\infty)$ in $[C^2(\bar{\Omega})]^2$ and thus $$F(u(\cdot,t_k),v(\cdot,t_k))\to F(U_\infty,V_\infty), \ \mathrm{as} \ t_k\to\infty,$$
%we have \eqref{s-2} by the fact $F(u,v)\leq F(u_0,v_0)$ from \eqref{3-3}.
 Then, the proof of Lemma \ref{S-P} is completed.
 \end{proof}
 With Lemmas \ref{S-B}, \ref{K-B} and  \ref{S-P} in hand, we now show the blowup of solutions under supercritical mass by the argument of contradiction.
\begin{lemma}\label{B-1}{\color{black} For any $M>\frac{4\pi}{\chi}$ and $M\not \in \{\frac{4\pi m}{\chi}: m\in\mathbb{N}^+\}$,  there exist initial value $(u_0,v_0)$ satisfying $\int_0u_0dx=M$ such that the corresponding solution of $\eqref{1-1}$ blows up.}
\end{lemma}
\begin{proof}Since $M\not \in \{\frac{4\pi m}{\chi}: m\in\mathbb{N}^+\}$, then by Lemma $\ref{S-B}$, we can find a  constant $K>0$ such that
 \begin{equation}\label{bu-1}
 F(U_\infty,V_\infty)\geq- K,
\end{equation}
where $(U_\infty,V_\infty)$ is a solution of the system $\eqref{s-1}$. For this constant $K>0$ chosen in \eqref{bu-1}, we can use Lemma $\ref{K-B}$ to show that there exists a  small $\varepsilon_0>0$  such that
$
F(U_{\varepsilon_0},V_{\varepsilon_0})<-K,
$ provided $M>\frac{4\pi}{\chi}$,
where
\begin{equation*}
V_{\varepsilon_0}(x)=\frac{1}{\chi}\left[\ln\left(\frac{\varepsilon_0^2}{(\varepsilon_0^2+\pi|x-x_0|^2)^2}\right)-\frac{1}{|\Omega|}\int_\Omega\ln\left(\frac{\varepsilon_0^2}{(\varepsilon_0^2+\pi|x-x_0|^2)^2}\right)dx\right],
\end{equation*}
and
\begin{equation*}
U_{\varepsilon_0}(x)=\frac{Me^{\chi V_{\varepsilon_0}(x)}}{\int_\Omega e^{\chi V_{\varepsilon_0}(x)}dx}.
\end{equation*}
Moreover, we can  check that $(U_{\varepsilon_0},V_{\varepsilon_0})\in [W^{1,\infty}(\Omega)]^2$ and $\int_\Omega U_{\varepsilon_0}(x)dx=M$. Then the solution of the system \eqref{1-1} with initial data $(u_0,v_0)=(U_{\varepsilon_0},V_{\varepsilon_0})$ must blow up.  In fact, suppose the solution $(u,v)$ of $\eqref{1-1}$  with the above $(u_0,v_0)$ is uniformly bounded in time, then from Lemma $\ref{S-P}$, we have $F(U_\infty,V_\infty)\leq F(u_0,v_0)<-K$, which combined with \eqref{bu-1} raises the following contradiction:
$$-K\leq F(U_\infty,V_\infty)\leq F(u_0,v_0)<-K.$$
Then the Lemma \ref{B-1} is proved.
\end{proof}

\subsubsection{Proof of Theorem $\ref{BS}$}
Theorem $\ref{BS}$ is a direct consequence of Lemma $\ref{BS-1}$ and Lemma $\ref{B-1}$.

\bigbreak

\noindent \textbf{Acknowledgement}.
We are grateful to the referees for several helpful comments improving our results. The research of H.Y. Jin was supported  by  the NSF of China No. 11871226  and  the Fundamental Research Funds for the Central Universities. The research of Z.A. Wang was supported by the Hong Kong RGC GRF grant No. 15303019 (Project ID  P0030816).


\begin{thebibliography}{99}



\bibitem{ADN-1959}S. Agmon, A. Douglis and L. Nirenberg, Estimates near the boundary for solutions ofelliptic partial differential equations satisfying general boundary conditions. {\it{I, Comm. Pure Appl. Math., }}12:623-727, 1959.
\bibitem{ADN-1964}S. Agmon, A. Douglis and L. Nirenberg, Estimates near the boundary for solutions of elliptic partial differential equations satisfying general boundary conditions. {\it{II, Comm. Pure Appl. Math.,}} 17:35-92, 1964.
    \bibitem{AY-Nonlinearity-2019}J. Ahn and C. Yoon, Global well-posedness and stability of constant equilibria in parabolic-elliptic chemotaxis systems without gradinet sensing. {\it Nonlinearity,} 32(4):1327-1351, 2019.
\bibitem{Alikakos-1979} N.D. Alikakos, $L^p$ bounds of solutions of reaction-diffusion equations.  {\it{Comm. Partial Differential Equations,}} 4:827-868, 1979.

\bibitem {A-DIE-1990}H. Amann, Dynamic theory of quasilinear parabolic equations. II. Reaction-diffusion systems. {\it Differ. Integral Equ.,} 3(1):13-75, 1990.

\bibitem{A-Book-1993} H. Amann, Nonhomogeneous linear and quasilinear elliptic and parabolic boundary value problems. {\it Function spaces, differential operators and nonlinear analysis.  Teubner-Texte zur Math., Stuttgart-Leipzig,} 133:9-126, 1993.
\bibitem{BBTW-M3AS-2015} N. Bellomo, A. Bellouquid, Y.S. Tao and M. Winkler, Towards a mathematical theory of Keller-Segel models of pattern formation in biological tissues. {\it Math. Mod. Meth. Appl. Sci.,} 25:1663-1763, 2015.
\bibitem{B-AMSA-1998}
P. Biler, Global solutions to some parabolic-elliptic systems of chemotaxis. {\it Adv. Math.Sci. Appl.,} 9:347-359,1999.

\bibitem{BP-CPAA-2012}A. Blancher and Ph. Lauren\c{c}ot, Finite mass self-similar blowing-up solutions of a chemotaxis system with non-linear diffusion. {\it Commun. Pure Appl. Anal.,} 11:47-60, 2012.

\bibitem{Burni-Chouhad}D. Burini and N. Chouhad, A multiscale view of nonlinear diffusion in biology: From cells to tissues. {\it Math. Mod. Meth. Appl. Sci}., 29: 791-823, 2019.



    \bibitem{Fu-PRL-2011} X. Fu, L.H. Tang, C. Liu, J.D. Huang, T. Hwa and P. Lenz, Stripe formation in bacterial system with density-suppressed motility. {\it Phys. Rev. Lett.}, 108:198102, 2012.
\bibitem{F-JMAA-2015}K. Fujie, Boundednss in a fully parabolic chemotaxis system with singular sensitivity. {\it J. Math. Anal. Appl.,} 424:675-684, 2015.

\bibitem{FWY-M2AS-2015}K. Fujie, M. Winkler and T. Yokota, Boundedness of solutions to parabolic-elliptic Keller-Segel systems with signal-depdent sensitivity. {\it Math. Meth. Appl. Sci.,} 38(6):1212-1224, 2015.


\bibitem{FS-DCDSB-2016}K. Fujie and T. Senba, Global existence and boundedness in  a parabolic-elliptic Keller-Segel system with general sensitivity. {\it Discrete Continuous Dyn. Syst. B,} 21:81-102, 2016.
\bibitem{FS-Nonlinearity-2016}K. Fujie and T. Senba, Global existence  and boundedness of radial solutions to a two dimensional fully parabolic chemotaxis system with general sensitivity. {\it Nonlinearity}, 29(8):2417-2450, 2016.

%\bibitem{Gilbard-Ttrudinger} D. Gilbarg and N. Trudinger, {\it Elliptic Partial Differential Equations of Second Order}, Springer-Verlag, New York, 2001.


\bibitem{HP-JMB-2009} T. Hillen and K.J. Painter, A user's guide to PDE models for chemotaxis. {\it J. Math. Biol.,} 58:183-217, 2009.

\bibitem{H-Review-1}
D. Horstmann, From 1970 until present: the Keller-Segel model in chemotaxis and its consequences. {\it I. Jahresber. Deutsch. Math.-Verein.,} 105(3):103-165, 2003.
\bibitem{H-Review-2}
D. Horstmann, From 1970 until present: the Keller-Segel model in chemotaxis and its consequences. {\it II. Jahresber. Deutsch. Math.-Verein.,} 106(2):51-69, 2004.
\bibitem{horstmann2001blow}
D. Horstmann and G. Wang, Blow-up in a chemotaxis model without symmetry
assumptions. {\it{Euro. J. Appl. Math.,}} 12:159-177, 2001.
  \bibitem{horstmann2001nonsymmetric}
D.~Horstmann, The nonsymmetric case of the $\mathrm{Keller}$-$\mathrm{Segel}$ model
  in chemotaxis: some recent results. {\em Nonlinear Differential Equations and Applications NoDEA}, 8(4):399--423, 2001.
  \bibitem{KPN-MA-2017}K. Ishige, Ph. Lauren\c{c}ot and N. Mizoguchi, Blow-up behavior of solutions to a degenerate parabolic-parabolic Keller-Segel system. {\it Math. Ann.,} 367:461-499, 2017.
    \bibitem{JKW-SIAP-2018} H.Y. Jin, Y.J. Kim and Z.A. Wang, Boundedness, stabilization, and pattern formation driven by density-suppressed motility. {\it SIAM J. Appl. Math.}, 78(3):1632-1657, 2018.
  \bibitem{JW-JDE-2016}H.Y. Jin and Z.A. Wang, Boundednesss, blowup and critical mass phenomenon in competing chemotaxis. {\it J. Differential Equations}, 260:162-196, 2016.
\bibitem{JWEJAM} H.Y. Jin and Z.A. Wang, Global dynamics and spatio-temporal patterns of predator-prey systems with density-dependent motion. {\it Euro. J. Appl. Math.}, in press, 2020.
\bibitem{Kareiva} P. Kareiva and G. Odell,  Swarms of predators exhibit {\color{black}``preytaxis"} if individual predators use area-restricted search. {\it  Amer. Nat.}, 130(2):233-270, 1987.
  \bibitem{KS-1971-JTB2} E.F. Keller and L.A. Segel, Models for chemtoaxis. {\it J. Theor. Biol.,} 30:225-234, 1971.
\bibitem{Lady-1968}O. Ladyzhenskaya, S. Solonnikov and N. Uralceva, {\it Linear and Quasilinear Equations of Parabolic Type} Providence, RI: American Mathematical Society, 1968.
\bibitem{La-M2AS-2016} J. Lankeit, A new approach toward boundedness in a two-dimensional parabolic chemotaxis system with singular sensitivity. {\it Math. Meth. Appl. Sci.,} 39:394-404, 2016.
    \bibitem{Liu-Science} C. Liu et. al, Sequential establishment of stripe patterns in an expanding cell population. {\it Science,} 334:238--241, 2011.

    \bibitem{Nagai-Funk}T. Nagai, T. Senba and K. Yoshida, Application of the Trudinger-Moser inequality to a parabolic system of chemotaxis. {\it{Funkcial. Ekvac.,}} 40:411-433, 1997.

         \bibitem{KA-FE-2001}K. Osaki and A. Yagi, Finite dimensional attractor for one-dimensional Keller-Segel equations. {\it Funkcial. Ekvac.,} 44:441-469, 2001.

\bibitem{SW-NA-2011}C. Stinner and M. Winkler, Global weak solutions in a chemotaxis system with large singular sensitivity. {\it Nonlinear Anal. Real Word Appl.,}12:3727-3740, 2011.
\bibitem{TW-JDE-2014} Y.S. Tao and M. Winkler, Energy-type estimates and global solvability in a two-dimensional chemotaxis-haptotaxis model with remodeling of non-diffusible attractant. {\it J. Differential Equations,} 257(3):784-815, 2014.
\bibitem{TW-2015-SIMA}Y. Tao and M. Winkler, Large time behavior in a multidimensional chemotaxis-haptotaxis model with slow signal diffusion. {\it SIAM J. Math. Anal.}, 47(6):4229-4250, 2015.
 \bibitem{TW-M3AS-2017} Y. Tao and M. Winkler, Effects of signal-dependent motilities in a Keller-Segel-type reaction-diffusion system. {\it Math. Mod. Meth. Appl. Sci.}, 27(19):1645-1683, 2017.
       \bibitem{WW} J. Wang and M. Wang, Boundedness in the higher-dimensional Keller-Segel model with signal-dependent motility and logistic growth. {\it J. Math. Phys}., 60(1):011507, 2019.
     \bibitem{Wang-Hillen} Z.A. Wang and T. Hillen, Classical solutions and pattern formation for a volume filling chemotaxis model. {\it Chaos}, 17:037108, 2007.
     \bibitem{WWY}J. Wang, Z.A. Wang and W. Yang, Uniqueness and convergence on equilibria of the Keller-Segel system with subcritical mass.
 {\it Comm. Partial Differential Equations}, 44:545-572, 2019.
%\bibitem{Wang-review} Z.A. Wang,  Mathematics of traveling waves in chemotaxis. {\it{Discrete Contin. Dyn. Syst-Series B.,}} 18(3): 601-641, 2013.
\bibitem{W-2010}
M. Winkler, Aggregation vs. global diffusive behavior in the higher-dimensional Keller-Segel model. {\it{J. Differential Equations,}} 248:2889-2905, 2010.
\bibitem{W-M2AS-2011}M. Winkler, Global solutions in a fully parabolic chemotaxis system with singular sensitivity. {\it Math. Meth. Appl. Sci.,} 34:176-190, 2011.

\bibitem{WY-NATM-2018}M. Winkler and T. Yokota, Stabilization in the logarithmic Keller-Segel system. {\it Nonlinear Anal. Theory Methods Appl.,} 170:123-141, 2018.
\bibitem{W-JMPA-2013} M. Winkler, Finite-time blow-up in the higher-dimensional parabolic-parabolic Keller-Segel system. {\it{J. Math. Pures Appl.,}} 100(5):748-767, 2013.
\bibitem{Wrzosek} D. Wrzosek, Model of chemotaxis with threshold density and singular diffusion.
    {\it Nonlinear Anal}., 73: 338-349, 2010.
\bibitem{YK-AAM-2017} C. Yoon and Y.J. Kim, Global existence and aggregation in a Keller-Segel model with Fokker-Planck diffusion. {\it Acta Appl. Math.,} 149:101-123, 2017.
\end{thebibliography}
\end{document}